\documentclass[letterpaper, 10 pt, conference]{ieeeconf}  % Comment this line out if you need a4paper

\IEEEoverridecommandlockouts                              % This command is only needed if 
\usepackage{cite}
\usepackage{amsmath,amssymb,amsfonts}
\usepackage{algorithm}
\usepackage{graphicx}
\usepackage{textcomp}
\usepackage{array}
\usepackage{multirow}
\usepackage{breqn}
\usepackage{algpseudocode}
\let\labelindent\relax
\usepackage{enumitem}
\usepackage{threeparttable}

\providecommand{\norm}[1]{\ensuremath{\left\lVert#1\right\rVert }}
\providecommand{\mnorm}[1]{\ensuremath{\left\lvert#1\right\rvert}}

\def\R{\mathbb{R}}
\def\Na{\mathbb{N}}
\def\O{\mathcal{O}}
\def\V{\mathcal{V}}
\def\D{\mathcal{D}}

\newtheorem{theorem}{\bfseries Theorem}
\newtheorem{lemma}{\bfseries Lemma}
\newtheorem{assumption}{\bfseries Assumption}

\title{\LARGE \bf On Convergence of the \\ Iteratively Preconditioned Gradient-Descent (IPG) Observer
}

\author{Kushal Chakrabarti$^1$, and Nikhil Chopra$^2$% <-this % stops a space
\thanks{$^1$ Tata Consultancy Services Research, Mumbai 400607, India}% <-this % stops a space
\thanks{$^2$ University of Maryland, College Park, Maryland 20742, U.S.A.}% <-this % stops a space
\thanks{Emails: {\em chakrabarti.k@tcs.com}, and {\em nchopra@umd.edu}}%
}

\begin{document}

\maketitle
\thispagestyle{empty}
\pagestyle{empty}

%%%%%%%%%%%%%%%%%%%%%%%%%%%%%%%%%%%%%%%%%%%%%%%%%%%%%%%%%%%%%%%%%%%%%%%%%%%%%%%%

\begin{abstract}
This paper considers the observer design problem for discrete-time nonlinear dynamical systems with sampled measurement data. Earlier, the recently proposed Iteratively Preconditioned Gradient-Descent (IPG) observer, a Newton-type observer, has been empirically shown to have improved robustness against measurement noise than the prominent nonlinear observers, a property that other Newton-type observers lack. However, no theoretical guarantees on the convergence of the IPG observer were provided. This paper presents a rigorous convergence analysis of the IPG observer for a class of nonlinear systems in deterministic settings, proving its local linear convergence to the actual trajectory. Our assumptions are standard in the existing literature of Newton-type observers, and the analysis further confirms the relation of the IPG observer with the Newton observer, which was only hypothesized earlier.
\end{abstract}
\section{Introduction}
\label{sec:intro}

This paper considers the observer design problem for a nonlinear discrete-time system with sampled measurements. Let $x_k \in \R^n$, $u_k \in \R^m$, and $y_k \in \R^p$ respectively denote the system state, the input to the system, and the observed measurement at the $k^{th}$ sampling instant. Then, for each $k \in \Na$, the dynamical system is described by
\begin{align}
    x_{k+1} = F(x_k, u_k), \, y_k & = h(x_k), \label{eqn:meas_eqn}
\end{align}
where the state-dynamics function $F: (\R^n,\R^m) \to R^n$ and the output function $h: \R^n \to \R^p$ are nonlinear. At each $k \geq N$, the observer has access to the past $N \in \mathbb{N}$ consecutive measurements $Y_k = \left[y_{k-N+1}^T,\ldots,y_k^T\right]^T \in \R^{N p}$ and $N-1$ inputs $U_k = \left[u_{k-N+1}^T,\ldots,u_{k-1}^T\right]^T \in \R^{(N-1) m}$. Considering a set of past measurements and inputs is common among Newton-type observers~\cite{moraal1995observer,biyik2006hybrida, biyik2006hybridb, hanba2008numerical} and moving horizon estimation problems~\cite{alessandri2010advances, alessandri2017fast}. 

There is a vast literature of other classes of nonlinear observers, including~\cite{astolfi2006global, niazi2023learning, khalil2014high, menini2021design, shim2015nonlinear, shamma1997approximate,karafyllis2009continuous, beikzadeh2016input, ahmed2020sampled, kashefi2022new}. A few notable ones are the exact linearization-based~\cite{astolfi2006global, niazi2023learning}, high-gain observers~\cite{khalil2014high, menini2021design}, and qDES observer framework~\cite{shim2015nonlinear}. The computational complexity of exact linearization-based observers increases with the dimension $n$ and the truncation order of a Taylor series approximation. Recently,~\cite{niazi2023learning} addressed these computational challenges using a learning-based technique and analyzed its robustness against process and system noise. However, the observer in~\cite{niazi2023learning} is for continuous-time dynamics and measurements. The high-gain observers require continuous-time measurements. Constructing a qDES observer, which considers noise in the system, relies on choosing suitable storage and class-$\mathcal{K}$ functions satisfying certain conditions~\cite{shim2015nonlinear}. Newton and approximate Newton observers are broadly applicable and converge fast in deterministic settings. However, these Newton-type observers result in a large steady-state error, even instability, in the presence of measurement noise. 

To address the noise-sensitivity of Newton-type observers, we recently proposed the IPG observer~\cite{chakrabarti2023ipg}, utilizing the idea of iterative pre-conditioning of the gradient-descent algorithm~\cite{chakrabarti2021accelerating}. The IPG observer is a Newton-type observer that estimates the states of the discrete-time nonlinear system~\eqref{eqn:meas_eqn} from a moving window of the past $N$ measurements and inputs. We introduce the following notation to review the IPG observer~\cite{chakrabarti2023ipg}. For simplicity, we define $F^{u_k}(x_k) = F(x_k,u_k)$ for each $k \in \Na$ and let $\circ$ denote composition of functions.
Recall the past $N$ measurements vector $Y_k = \left[y_{k-N+1}^T,\ldots,y_k^T\right]^T$. Then, from~\eqref{eqn:meas_eqn}, for each $k \geq N$,
\scalebox{0.95}{\parbox{1\linewidth}{
\begin{align}
    Y_k \hspace{-0.2em} = \hspace{-0.3em} \begin{bmatrix} h(x_{k-N+1}) \\ \vdots \\ h \circ F^{u_{k-1}} \circ \ldots \circ F^{u_{k-N+1}}(x_{k-N+1}) \end{bmatrix} \hspace{-0.3em} =: \hspace{-0.2em} H^{U_k}(x_{k-N+1}), \hspace{-0.2em} \label{eqn:meas_vec_full}
\end{align}
}}
where $H^{U_k} : \R^n \to \R^{N p}$ is known as the {\em observability mapping} of~\eqref{eqn:meas_eqn}. The IPG observer is iterative, wherein at each sampling instant $k \geq N$, it applies $d$ IPG iterations indexed by $i=0,\ldots,d-1$. At each $i$ and $k$, it maintains an estimate $w_k^{(i)} \in \R^n$ of $x_{k-N+1}$, an estimate $\hat{x}_k \in \R^n$ of the true state $x_k$, and a matrix $K_k^{(i)} \in \R^{n \times n}$. Before initiating the observer, it chooses an initial estimate $w_N^{(0)} \in \R^n$, an initial pre-conditioner $K_N^{(0)} \in \R^{n \times n}$, and the parameter $d \in \Na$. It executes the following steps at each $k \geq N$.
If $n=Np$, then in the same iteration $i=0,\ldots,d-1$,
\begin{align}
    K_k^{(i+1)} & = K_k^{(i)} - \alpha^{(i)} \left(H^{U_k}_x (w_k^{(i)}) K_k^{(i)} - I\right), \label{eqn:K_sq} \\
    w_k^{(i+1)} & = w_k^{(i)} - \delta^{(i)} K_k^{(i)} \left(H^{U_k}(w_k^{(i)}) - Y_k\right), \label{eqn:ipg_sq}
\end{align}
where $H^U_x(\cdot) := \frac{\partial H^U}{\partial x} (\cdot)$.
From $w_k^{(d)}$, after the above $d$ IPG iterations, the estimate $\hat{x}_k$ of $x_k$ is obtained by propagating:
\begin{align}
    \hat{x}_k & = F^{u_{k-1}} \circ \ldots \circ F^{u_{k-N+1}}(w_k^{(d)}). \label{eqn:newton_est}
\end{align}
The initialization of the IPG iterations in the next sampling period $k+1$ is set by propagating $w_k^{(d)}$ forward:
\begin{align}
    w_{k+1}^{(0)} & = F^{u_{k-N+1}}(w_k^{(d)}), \label{eqn:newton_init} \\
    K_{k+1}^{(0)} & = K_k^{(d)}. \label{eqn:ipg_init}
\end{align}
We assume $n = Np$, i.e., $H_x^{U}$ is a square matrix. Comparing the fixed point of~\eqref{eqn:K_sq}-\eqref{eqn:ipg_sq} above with (8)-(9) of~\cite{chakrabarti2023ipg}, the extension of our result to the case of $n \neq Np$ is straightforward, by replacing $(H^{U_k}_x)^{-1}$ with its pseudo-inverse.

Our work in~\cite{chakrabarti2023ipg} (i) hypothesized that under appropriate conditions, the IPG observer converges to the Newton observer as the sampling instant $k \to \infty$ and the number of iterations $i \to \infty$ within each $k$, and (ii) empirically showed its improved robustness against measurement noise than the prominent nonlinear observers that apply to general nonlinear functions $F,h$. \underline{However, no theoretical guarantees were given in~\cite{chakrabarti2023ipg}.}

{\bf Summary of Our Contributions}:
Therefore, this work presents a formal convergence analysis of the IPG observer for a class of nonlinear discrete-time systems~\eqref{eqn:meas_eqn} without process or measurement noise. This class of systems~\eqref{eqn:meas_eqn} is identified by our set of assumptions, presented later in Section~\ref{sub:assump}, which are standard for Newton-type observers~\cite{moraal1995observer, biyik2006hybrida, biyik2006hybridb}. Notably, we prove local {\em linear} convergence of the IPG observer in Section~\ref{sec:conv}. We note that the {\em linear} convergence rate of the Newton observer and approximate Newton observers~\cite{moraal1995observer, biyik2006hybrida, biyik2006hybridb} are also local and do not study the impact of measurement noise. The global asymptotic convergence of the Newton observer has been proved only when the observer includes continuous-time dynamics of the auxiliary variable. From our convergence analysis, we also prove the hypothesis in~\cite{chakrabarti2023ipg} that under appropriate assumptions, the IPG observer converges to Newton observer as the sampling instant $k \to \infty$ and the number of iterations $i \to \infty$ within each $k$. A detailed empirical comparison with prominent observers can be found in our prior work~\cite{chakrabarti2023ipg}.

% We compare the performance of the IPG observer against that of existing general observers for nonlinear systems on a nonlinear spring-mass-damper system. A detailed empirical comparison with these observers on additional examples can be found in our prior work~\cite{chakrabarti2023ipg}; hence, not included here.

% $\left((\frac{\partial H^{U_k}}{\partial w})^T \frac{\partial H^{U_k}}{\partial w}\right)^{-1} (\frac{\partial H^{U_k}}{\partial w})^T$.
\section{Assumptions and Prior Results}
\label{sec:prelim}

% In this section, we introduce our main assumptions and review a pivotal prior result of the iteratively pre-conditioned gradient-descent (IPG) optimizer.

\subsection{Assumptions}
\label{sub:assump}

% To present our theoretical results, we make the following assumptions.

\begin{assumption} \label{assump:invar}
There exists a convex and compact subset $\O \subset \R^n$ and a compact subset $\V \subset \R^m$ such that for each $x \in \O$ there exists $u \in \V$ such that $F(x,u) \in \O$. Moreover, the controls are applied so that $F(x,u) \in \O$.
\end{assumption}

\begin{assumption} \label{assump:obsv}
There exists an integer $N \in [1,n]$ such that the set of algebraic equations~\eqref{eqn:meas_vec_full} has (i) has $n = N p$; (ii) is uniformly $N$-observable with respect to (w.r.t.) $\O, \V^{N-1} := \underbrace{\V \times \ldots \times \V}_{N-1 \, \text{times}}$; (iii) satisfies the $N$-observability rank condition w.r.t. $\O,\V^{N-1}$.
\end{assumption}

\begin{assumption} \label{assump:diff}
$F$ and $h$ are twice continuously differentiable w.r.t. its first argument over $\R^n$.
\end{assumption}

% \begin{assumption} \label{assump:F_lips}
% $F^u$ is $L$-Lipschitz continuous over $\R^n$, i.e., $\norm{F^u(x) - F^u(y)} \leq L \norm{x-y}$ for any $x,y \in \R^n$.
% \end{assumption}

% \begin{assumption} \label{assump:H_lips}
% $H$ is $l$-Lipschitz continuous over $\R^n$, i.e., $\norm{H(x) - H(y)} \leq l \norm{x-y}$ for any $x,y \in \R^n$.
% \end{assumption}

% \begin{assumption} \label{assump:gamma_lips}
% $\frac{\partial H}{\partial x}$ is $\gamma$-Lipschitz continuous over $\R^n$, i.e., $\norm{\frac{\partial H}{\partial x}(x) - \frac{\partial H}{\partial x}(y)} \leq \gamma \norm{x-y}$ for any $x,y \in \R^n$.
% \end{assumption}

% \begin{assumption} \label{assump:delH_lips}
% $(\frac{\partial H}{\partial x})^{-1}$ is $L_2$-Lipschitz continuous over $\O$, i.e., $\norm{(\frac{\partial H}{\partial x}(x))^{-1} - (\frac{\partial H}{\partial x}(y))^{-1}} \leq L_2 \norm{x-y}$ for any $x,y \in \O$.
% \end{assumption}

% \begin{assumption} \label{assump:hess_bd}
% The largest eigenvalue $\lambda_{\max} \left[\frac{\partial H}{\partial x} (x) \right]$ is uniformly bounded by a positive real constant $\Lambda$ for any $x \in \R^n$.
% \end{assumption}

\begin{assumption} \label{assump:eigen}
Eigenvalues of $H^U_x$ are positive in $\R^n$.
\end{assumption}

Next, we provide the justifications for the above assumptions and discuss their implications, which will be useful in presenting our main result. Assumptions~\ref{assump:invar}-\ref{assump:diff} are standard. Particularly, Assumption~\ref{assump:invar} implies that the system states are bounded and lie within a controlled-invariant subspace, which is standard for Newton-type observers~\cite{moraal1995observer, biyik2006hybrida, biyik2006hybridb}. Assumptions~\ref{assump:obsv}-\ref{assump:diff} are also from~\cite{moraal1995observer, biyik2006hybrida, biyik2006hybridb}. Specifically, (ii) and (iii) of Assumption~\ref{assump:obsv} are jointly equivalent to: for each $U_k \in \V^{N-1}$, $H^{U_k}: \O \to \R^{Np}$ is an {\em injective immersion}~\cite{moraal1995observer}. Assumption~\ref{assump:obsv} leads to invertible $H^U_x$ over $\O$, which is also assumed in~\cite{biyik2006hybrida, biyik2006hybridb}. As discussed earlier in Section~\ref{sec:intro}, extending our result to the non-square case is straightforward. 

\begin{lemma} \label{lem:assump}
Let $W \subset \R^n$ be any compact and convex set. Under Assumptions~\ref{assump:invar}-\ref{assump:diff}, the following hold.
\begin{itemize}[leftmargin=*]
    \item $\exists L > 0$ such that $\norm{F^u(x) - F^u(y)} \leq L \norm{x-y}$ for any $x,y \in W$.
    \item $\exists l, \gamma > 0$ such that $\norm{H^U(x) - H^U(y)} \leq l \norm{x-y}$ and $\norm{H^U_x(x) - H^U_x(y)} \leq \gamma \norm{x-y}$ for any $x,y \in W$.
    \item The largest eigenvalue $\lambda_{\max} \left[H^U_x (x) \right] \leq \Lambda$ for $x \in W$.
    \item $\exists L_2 > 0$ such that $\norm{(H^U_x(x))^{-1} - (H^U_x(y))^{-1}} \leq L_2 \norm{x-y}$ for any $x,y \in \O$.
\end{itemize}
\end{lemma}

\begin{proof}
Assumption~\ref{assump:diff} implies that $H^U$ and $F$ are twice continuously differentiable over $\R^n$. So, Using Lemma~3.2 of~\cite{khalil1996nonlinear}, $H^U$, $H^U_x$, and $F^u$ are Lipschitz continuous on $W$. 

Since $H^U_x$ is continuous, it is bounded over a compact set $W$. So, its eigenvalues are bounded over $W$. 

Under Assumptions~\ref{assump:invar}-\ref{assump:obsv}, upon differentiating both sides of $H^U_x (H^U_x)^{-1} = I$, we have $\frac{\partial}{\partial x_i} (H^U_x)^{-1} = - (H^U_x)^{-1} \frac{\partial}{\partial x_i} H^U_x (H^U_x)^{-1}$ over $\O$, $\forall i$. Since $(H^U_x)^{-1}$ exists over $\O$ and $H^U$ is twice continuously differentiable, the above implies that $\frac{\partial}{\partial x} (H^U_x)^{-1}$ is continuous over $\O$. So, using Lemma~3.2 of~\cite{khalil1996nonlinear}, $(H^U_x)^{-1}$ is Lipschitz over $\O$.
\end{proof}

Assumption~\ref{assump:eigen} is additional compared to~\cite{moraal1995observer, biyik2006hybrida, biyik2006hybridb}. We explain the reason behind this assumption below. Later, in Section~\ref{sec:conv}, we discuss its relaxation.
Note that the convergence guarantee of the IPG algorithm~\cite{chakrabarti2021accelerating} assumes convex objective functions, which is equivalent to a positive semi-definite Jacobian $H^U_x$. So, to utilize the result from~\cite{chakrabarti2021accelerating}, we need $\lambda_{\min} \left[H^U_x\right] \geq 0$. In addition, $H^U_x$ is invertible under the observability Assumption~\ref{assump:obsv}, as mentioned above. So, we have assumed $\lambda_{\min} \left[H^U_x\right] > 0$ in Assumption~\ref{assump:eigen}. A similar assumption can be found in moving horizon estimation problems~\cite{alessandri2017fast} and its relation with Assumption~\ref{assump:obsv} in~\cite{hanba2010further}.

\subsection{Prior Result on Convergence of IPG Optimizer}
\label{sub:prior}

We review below in Lemma~\ref{lem:ipg_conv} a prior result from~\cite{chakrabarti2021accelerating} that is pivotal for our key result. Consider the optimization problem $x^* \in X^* = \arg \min_{z \in \R^d} f(x)$. 
\begin{lemma}[Theorem~2 and Proof of Theorem~1 in~\cite{chakrabarti2021accelerating}] \label{lem:ipg_conv}
Suppose that the following conditions hold:
\begin{itemize}
    \item $\mnorm{\min_{x \in \R^d} f(x)} < \infty$;
    \item $f$ is strongly convex and twice continuously differentiable over a compact convex domain $\D \subseteq \R^d$;
    \item the gradient $\nabla f$ is $l_f$-Lipschitz continuous over $\D$;
    \item the Hessian $\nabla^2 f$ is $\gamma_f$-Lipschitz continuous over $\D$;
\end{itemize}
Consider the centralized counterpart of the IPG algorithm in~\cite{chakrabarti2021accelerating}, with parameters $\beta = 0$, $\delta = 1$, and $\alpha(t) < \frac{1}{\lambda_{\max} \left[\nabla^2 f(x(t)) \right]}$ for each iteration $t \geq 0$. Then, 
$\rho_f = \sup_{t \geq 0} \norm{I - \alpha(t) \nabla^2 f(x(t))} < 1$. Let the initial estimate $x(0) \in \D$ and the initial preconditioner $K(0) \in \R^{d \times d}$ be chosen such that 
\begin{align}
    \frac{\eta_f \gamma_f}{2} \hspace{-0.2em} \norm{x(0) - x^*} \hspace{-0.2em} + \hspace{-0.2em} l_f \hspace{-0.2em} \norm{K(0) - (\nabla^2 f(x^*))^{-1}} \hspace{-0.2em} \leq \hspace{-0.2em} \frac{1}{2\mu_f}, \label{eqn:ipg_initial}
\end{align}
where $1 < \mu_f < \frac{1}{\rho_f}$ and $\eta_f = \norm{(\nabla^2 f(x^*))^{-1}}$. If for $t \geq 0$, 
$$\alpha(t) < \min \left\{\frac{1}{\lambda_{\max} \left[\nabla^2 f(x(t)) \right]},\frac{\mu_f^t (1-\mu_f \rho_f)}{2l_f (1-(\mu_f \rho_f)^{t+1})} \right\},$$ then we obtain that, for each iteration $t \geq 0$,
\begin{align*}
    & \norm{x(t+1) - x^*} \leq \frac{1}{\mu_f} \norm{x(t) - x^*}, \\
    & \norm{K(t+1) - (\nabla^2 f(x^*))^{-1}} \leq \rho_f^{t+1} \norm{K(0) - (\nabla^2 f(x^*))^{-1}} \\
    & + \eta_f \gamma_f \alpha(t) (\norm{x(t)-x^*} + \ldots + \rho_f^t \norm{x(0)-x^*}).
\end{align*}
\end{lemma}
% ~\\

The second condition in Lemma~\ref{lem:ipg_conv} is equivalent to  $\nabla^2 f$ having positive eigenvalues and implies that the set of minima $X^*$ is a singleton. The above result implies 
{\em linear} convergence of the IPG algorithm in~\cite{chakrabarti2021accelerating} for iteratively minimizing a cost function $f$. A sufficient condition of this convergence guarantee is local initialization of $x(0)$ and $K(0)$ within a neighborhood of $x^*$ and $\nabla^2 f(x^*))^{-1}$, where the radius of the neighborhood has a trade-off with the convergence rate $\frac{1}{\mu_f}$ (see~\eqref{eqn:ipg_initial}). Moreover, $\rho_f$ determines the convergence rate of the pre-conditioner $K(t)$ and the convergence rate $\frac{1}{\mu_f}$ of $x(t)$ is lower-bounded by $\rho_f$. In other words, the larger is stepsize $\alpha(t)$, the faster is the convergence of $K(t)$ (smaller $\rho_f$), which leads to faster convergence of $x(t)$ (smaller lower-bound of $\mu_f$).

Since, at each sampling instant $k \geq N$, the IPG observer~\eqref{eqn:K_sq}-\eqref{eqn:ipg_init} executes $d$ number of IPG iterations from~\cite{chakrabarti2021accelerating} and the convergence guarantee of IPG iterations is presented in Lemma~\ref{lem:ipg_conv}, it makes sense to analyze IPG observer's convergence by invoking Lemma~\ref{lem:ipg_conv} at each time instant $k$. Now, to invoke the result of Lemma~\ref{lem:ipg_conv}, we need to show that the conditions of Lemma~\ref{lem:ipg_conv} hold true at each sampling instant $k$ of the IPG observer~\eqref{eqn:K_sq}-\eqref{eqn:ipg_init}. However, from the fixed point of IPG observer dynamics~\eqref{eqn:K_sq}-\eqref{eqn:ipg_sq}, $K_k^{(d)}$ is a finite-time estimate of $(H_x^U)^{-1}(x_{k-N+1})$ and $K_{k+1}^{(0)} = K_k^{(d)}$ is the initial estimate of $(H_x^U)^{-1}(x_{k-N+2})$. This leads to non-descent direction in~\eqref{eqn:ipg_sq} at the beginning $i=0$ of each sampling instant $k$ due to the difference $\norm{(H_x^U)^{-1}(x_{k-N+2}) - (H_x^U)^{-1}(x_{k-N+1})}$. Thus, the key challenge here lies in proving the local initialization~\eqref{eqn:ipg_initial} at each sampling instant $k$ due to (i) premature termination of IPG iterations after a finite $d$ number of iterations, and (ii) the initial value of the pre-conditioner $K_{k+1}^{(0)}$ at any sampling instant $k+1$ not necessarily resulting in a descent direction of $w_k^{(i)}$ as $K_{k+1}^{(0)}$ is initialized without utilizing any information of the system~\eqref{eqn:meas_eqn}. Our analysis of IPG observer in the next section addresses this challenge while invoking Lemma~\ref{lem:ipg_conv}.
\section{Convergence of IPG Observer}
\label{sec:conv}

In this section, we present our key result on the convergence of IPG observer, described in~\eqref{eqn:K_sq}-\eqref{eqn:ipg_init}, for the case $n=Np$. Following~\cite{moraal1995observer}, to reduce the notational burden, we present our analysis for the special case of~\eqref{eqn:meas_eqn} with no inputs. Due to compactness in Assumption~\ref{assump:invar}, the steps for the system with inputs are the same. Accordingly, we remove the superscripts $u, u_k, U, U_k$ from all the notations ($F^{u_k}, H^{U_k}$ etc.) related to inputs in this section. For the case of no inputs, Assumption~\ref{assump:invar} becomes $F(\O) \subset \O$ and~\eqref{eqn:meas_vec_full} becomes
\begin{align}
    Y_k & = \begin{bmatrix} h(x_{k-N+1}) \\ \vdots \\ h \circ F^{(N-1)} (x_{k-N+1}) \end{bmatrix}
    =: H(x_{k-N+1}). \label{eqn:meas_vec}
\end{align}

We introduce the following notation. For the sampling instant $k=N$, we denote $\rho_{N} := \sup_{i \geq 0} \norm{I - \alpha^{(i)} H_x (w_{N}^{(i)})}$,
which characterize the rate of convergence of the pre-conditioner $K_N^{(i)}$ in the IPG iterations within the $N$-th sampling instant. Similarly, we denote $\rho = \sup_{i\geq 0, k \geq 0} \norm{I - \alpha^{(i)} H_x (w_{N+k}^{(i)})}$,
which characterize the supremum of rates of convergence of the pre-conditioner $K_{k}^{(i)}$ in the IPG iterations within the $k$-th sampling instant, over all $k \geq N$. Under Assumption~\ref{assump:obsv}, $(H_x)^{-1}$ is continuous; and under Assumption~\ref{assump:invar}, state trajectory of~\eqref{eqn:meas_eqn} is a subset of the compact set $\O$. Thus, we can define the finite quantities
\[\eta = \sup_{k \geq 1} \norm{(H_x (F(x_{k})))^{-1}}, C_k = \norm{x_{k+1} - F(x_{k+1})}.\]
We denote the initial estimation error $\delta = \norm{w_N^{(0)} - x_1}$
and estimation error after $d$ iterations at first sampling instant
$\overline{\delta} = \norm{w_N^{(d)} - x_1}$.

The following lemma states that the supremum of convergence rates of $K_{k}^{(i)}$ in the IPG iterations within $k$-th sampling instant is less than one, which will be essential in proving {\em linear} convergence of $\hat{x}_k$ in our main result afterwards.

\begin{lemma} \label{lem:rho}
Suppose that Assumptions~\ref{assump:invar}-\ref{assump:eigen} hold. Consider the IPG observer in~\eqref{eqn:K_sq}-\eqref{eqn:ipg_init}, with $\alpha^{(i)} < \frac{1}{\Lambda}$ for iterations $i = 0,\ldots,d-1$ at each $k\geq N$. Then, $\rho < 1$.
\end{lemma}

\begin{proof}
Under Assumption~\ref{assump:eigen}, $\lambda_{\min} \left[H_x (x) \right] > 0, \, \forall x \in \R^n$. Under Assumptions~\ref{assump:invar}-\ref{assump:diff}, from Lemma~\ref{lem:ipg_conv} we have $\lambda_{\max} \left[H_x (x) \right] \leq \Lambda$. So, $\alpha^{(i)} < \frac{1}{\Lambda}$ implies that $\alpha^{(i)} < \frac{1}{\lambda_{\max} \left[H_x (x) \right]}, \, \forall x \in \R^n$. Then, it follows that $\norm{I - \alpha^{(i)} H_x (w_{N+k}^{(i)})} < 1, \, \forall i,k$. The proof is complete.
\end{proof}

\begin{theorem}
\label{thm:conv}
Consider the dynamical system~\eqref{eqn:meas_eqn} with no inputs, and IPG observer in~\eqref{eqn:K_sq}-\eqref{eqn:ipg_init} with $\alpha^{(i)} < \frac{1}{\Lambda}$. Suppose Assumptions~\ref{assump:invar}-\ref{assump:eigen} and the following conditions hold true:
\begin{enumerate}[label = (\roman*), wide, labelwidth=!, labelindent=0pt]
    \item number of iterations $d \geq \max\{1, 1 + \log_{\mu} L, (N-1) \log_{\mu} L\}$;
    \item $w_N^{(0)} \in \R^n$ and $K_N^{(0)} \in \R^{n \times n}$ be initialized such that
    \begin{align}
        \frac{\eta \gamma \delta}{2} + l \norm{K_{N}^{(0)} - (H_x (x_{1}))^{-1}} \leq \frac{1}{2\mu}, \label{eqn:init_cond_thm}
    \end{align}
    where $1 < \mu < \frac{1}{\rho}$;
    \item the step-size parameter 
    \begin{align}
        \alpha^{(i)} < \min \left\{\frac{1}{\Lambda}, \min\left\{\varrho, D_2\right\} \frac{\mu^{i} (1-\mu \rho)}{2l (1-(\mu \rho)^{i+1})} \right\}, \forall i \geq 0, \nonumber
    \end{align}
    where $\overline{\delta} < \frac{\delta}{L}$, $\varrho < (1-\rho)$, $D_2 < \frac{\eta \gamma (1 - L \frac{\overline{\delta}}{\delta})}{2l}$ and $\frac{1}{\mu} < 1 - \varrho$;
    \item for $k \geq 1$: 
    \begin{align}
        & \frac{l C_k L_2}{L \overline{\delta}} \leq \frac{1}{2\mu L \overline{\delta}}(1-\rho^{d}) +  \frac{1}{\mu^{k-1}} (\rho^{d} \frac{\eta \gamma}{2} - \varrho \frac{\eta \gamma}{2} - \frac{\eta \gamma}{2} \frac{1}{\mu}); \nonumber
    \end{align}
    \item for $k=0$:
    \begin{align}
        & l C_0 L_2
        \leq (1-\rho_N^{d}) (\frac{1}{2\mu} - \frac{\eta \gamma \delta}{2}) + \delta (\frac{\eta \gamma}{2} - \frac{\eta \gamma L \overline{\delta}}{2\delta} - l D_2); \nonumber
    \end{align}
\end{enumerate}
Then, we obtain
\begin{align}
    \norm{\hat{x}_{N+k} - x_{N+k}} & \leq \frac{1}{\mu^{k}} \norm{w_{N}^{(0)} - x_{1}}, \forall k \geq 0. \label{eqn:obsv_conv}
\end{align}
\end{theorem}

\begin{proof}
We prove the result using the principle of induction. The proof is divided into three parts. In the first part, we assume a certain statement on the convergence of $w_{N+k}^{(d)}$ is true for sampling instants $1,\ldots,k$ and prove it for sampling instant $k+1$. In the second part, we prove the same statement for the initial sampling instant $k=1$. In the final part, we use the above convergence result on $w_{N+k}^{(d)}$ to prove convergence of the state estimates $\hat{x}_{N+k}$.

{\bf Part 1}: Suppose that the following is true for $s=1,\ldots,k$:
\begin{align}
    \frac{\eta \gamma L \overline{\delta}}{2} \frac{1}{\mu^{s-1}} + l \norm{K_{N+s}^{(0)} - (H_x (F(x_s)))^{-1}} & \leq \frac{1}{2\mu}, \label{eqn:ind_clm_1} \\
    \norm{w_{N+s}^{(d)} - x_{s+1}} \leq \frac{1}{\mu} \norm{w_{N+s-1}^{(d)} - x_{s}}. \label{eqn:ind_clm_2}
\end{align}
We will prove that~\eqref{eqn:ind_clm_1}-\eqref{eqn:ind_clm_2} hold for $s=k+1$. As discussed in Section~\ref{sec:prelim}, our analysis relies on invoking the initial condition~\eqref{eqn:ipg_initial} at $i=0$ of each sampling instant $k \geq N$, which is the challenging part. Eq.~\eqref{eqn:ind_clm_1} is a condition similar to~\eqref{eqn:ipg_initial} but adjusted to our IPG observer, where it will be shown later that the first term $\frac{\eta \gamma L \overline{\delta}}{2 \mu^{k-1}}$ is an upper bound of $\frac{\eta \gamma}{2} \norm{w_{N+k}^{(0)} - x_{k+1}}$. We use this bound, instead of the actual $\frac{\eta \gamma}{2} \norm{w_{N+k}^{(0)} - x_{k+1}}$, to utilize the ``expected'' convergence rate $\frac{1}{\mu}$ over the sampling instants $k \geq N$ (or, $N+s$ with $s \geq 0$). To prove that $\frac{1}{\mu}$ is indeed the convergence rate, we have included~\eqref{eqn:ind_clm_2} in our assumption~\eqref{eqn:ind_clm_1}-\eqref{eqn:ind_clm_2} which we will prove jointly using induction. We proceed as follows. 

Our first step is to prove that $\frac{L \overline{\delta}}{\mu^{k-1}}$ is an upper bound of $\norm{w_{N+k}^{(0)} - x_{k+1}}$. Upon iterating~\eqref{eqn:ind_clm_2} from $s$ to $1$,
\begin{align}
    \norm{w_{N+s}^{(d)} - x_{s+1}} \leq \frac{1}{\mu^s} \norm{w_{N}^{(d)} - x_{1}} = \frac{\overline{\delta}}{\mu^s}. \label{eqn:wk_1}
\end{align}
For $s=k-1$, then
\begin{align}
    \norm{w_{N+k-1}^{(d)} - x_{k}} \leq \frac{\overline{\delta}}{\mu^{k-1}}. \label{eqn:wk_d}
\end{align}
Upon substituting from~\eqref{eqn:newton_init} and~\eqref{eqn:meas_eqn},
\begin{align*}
    \norm{w_{N+k}^{(0)} - x_{k+1}} = \norm{F(w_{N+k-1}^{(d)}) - F(x_{k})}.
\end{align*}
Using Lemma~\ref{lem:assump}, from above we have
\begin{align*}
    \norm{w_{N+k}^{(0)} - x_{k+1}} \leq L \norm{(w_{N+k-1}^{(d)} - x_{k})}.
\end{align*}
Upon substituting above from~\eqref{eqn:wk_d},
\begin{align}
    \norm{w_{N+k}^{(0)} - x_{k+1}} < \frac{L \overline{\delta}}{\mu^{k-1}}. \label{eqn:wk_0}
\end{align}

Now, we substitute this upper bound above in~\eqref{eqn:ind_clm_1}, and obtain a condition similar to~\eqref{eqn:ipg_initial} for the IPG observer at iteration $i=0$ of sampling instant $N+k$. Specifically,
the above and~\eqref{eqn:ind_clm_1} at $s = k$ implies that
\begin{align*}
    & \frac{\eta \gamma}{2} \norm{w_{N+k}^{(0)} - x_{k+1}} + l \norm{K_{N+k}^{(0)} - (H_x (F(x_k)))^{-1}} \leq \frac{1}{2\mu}.
\end{align*}

Having proved the condition above, we are ready to invoke Lemma~\ref{lem:ipg_conv}. To show that all the conditions of Lemma~\ref{lem:ipg_conv} hold, we denote $B(y,r) = \{x \in \R^n : \norm{y-x} \leq r\}$ for $y \in \R^n$ and $r > 0$ and $\rho_{N+k} := \sup_{i \geq 0} \norm{I - \alpha^{(i)} H_x (w_{N+k}^{(i)})}$. From Lemma~\ref{lem:rho}, $\rho_{N+k} < 1$.
Thus, under Assumptions~\ref{assump:obsv}-\ref{assump:eigen} and Lemma~\ref{lem:assump} with $W = \D = B(x_{k+1},\frac{1}{\mu \eta \gamma})$, if $\alpha^{(i)} < \min \left\{\frac{1}{\Lambda},\frac{\mu^{i} (1-\mu \rho_{N+k})}{2l (1-\mu \rho_{N+k})^{i+1}}\right\}$, the conditions of Lemma~\ref{lem:ipg_conv} holds. So, from Lemma~\ref{lem:ipg_conv} we have for $i=d-1$,
\begin{align}
    & \norm{K_{N+k}^{(d)} \hspace{-0.3em} - \hspace{-0.3em} (H_x (x_{k+1}))^{-1}} \hspace{-0.3em} \leq \hspace{-0.3em} \rho_{N+k}^{d} \hspace{-0.3em} \norm{K_{N+k}^{(0)} \hspace{-0.3em} - \hspace{-0.3em} (H_x (x_{k+1}))^{-1}} \nonumber \\
    & + \eta \gamma \alpha^{(d-1)} (\norm{w_{N+k}^{(d-1)} - x_{k+1}} + \rho_{N+k} \norm{w_{N+k}^{(d-2)} - x_{k+1}} \nonumber \\
    & + \ldots + \rho_{N+k}^{d-1} \norm{w_{N+k}^{(0)} - x_{k+1}}), \label{eqn:K_1} \\
    & \norm{w_{N+k}^{(i)} - x_{k+1}} \leq \frac{1}{\mu} \norm{w_{N+k}^{(i-1)} - x_{k+1}}, \forall i\geq 0. \label{eqn:w_1}
\end{align}

The rest of Part-I involves algebraic manipulation and using the results derived above to prove~\eqref{eqn:ind_clm_1}-\eqref{eqn:ind_clm_2} for $s=k+1$. First, we obtain a bound for $\norm{K_{N+s}^{(0)} - (H_x (F(x_s)))^{-1}}$ at $s=k+1$, using~\eqref{eqn:K_1}-\eqref{eqn:w_1}, as follows.
Upon iterating~\eqref{eqn:w_1},
\begin{align*}
    \norm{w_{N+k}^{(i)} - x_{k+1}} \leq \frac{1}{\mu^{i}} \norm{w_{N+k}^{(0)} - x_{k+1}}, \forall i\geq 0.
\end{align*}
To replace each $\norm{w_{N+k}^{(i)} - x_{k+1}}$ in~\eqref{eqn:K_1} with their upper bound in terms of $\norm{w_{N+k}^{(0)} - x_{k+1}}$, i.e., $i=0$, we substitute from above in~\eqref{eqn:K_1}:
\begin{align}
    & \norm{K_{N+k}^{(d)} \hspace{-0.3em} - \hspace{-0.3em} (H_x (x_{k+1}))^{-1}} \hspace{-0.3em} \leq \hspace{-0.3em} \rho_{N+k}^{d} \norm{K_{N+k}^{(0)} \hspace{-0.3em} - \hspace{-0.3em} (H_x (x_{k+1}))^{-1}} \nonumber \\
    & + \eta \gamma \alpha^{(d-1)} \frac{1-(\rho_{N+k} \mu)^{d}}{\mu^{d-1}(1-\rho_{N+k} \mu)} \norm{w_{N+k}^{(0)} - x_{k+1}}. \nonumber
\end{align}
Upon substituting above from~\eqref{eqn:ind_clm_1},
\begin{align}
    & \norm{K_{N+k}^{(d)} - (H_x (x_{k+1}))^{-1}}  \leq \frac{\rho_{N+k}^{d}}{l} (\frac{1}{2\mu} - \frac{\eta \gamma L \overline{\delta}}{2} \frac{1}{\mu^{k-1}}) \nonumber \\
    & + \eta \gamma \alpha^{(d-1)} \frac{1-(\rho_{N+k} \mu)^{d}}{\mu^{d-1}(1-\rho_{N+k} \mu)} \norm{w_{N+k}^{(0)} - x_{k+1}}. \label{eqn:K_3}
\end{align}
% Upon substituting from~\eqref{eqn:newton_init} and~\eqref{eqn:meas_eqn},
% \begin{align*}
%     \norm{w_{N+k}^{(0)} - x_{k+1}} = \norm{F(w_{N+k-1}^{(d)}) - F(x_k)}.
% \end{align*}
% Under Assumption~\ref{assump:F_lips}, from above we have
% \begin{align*}
%     \norm{w_{N+k}^{(0)} - x_{k+1}} \leq L \norm{(w_{N+k-1}^{(d)} - x_k)}.
% \end{align*}
% Upon substituting above from~\eqref{eqn:wk_1},
% \begin{align*}
%     \norm{w_{N+k}^{(0)} - x_{k+1}} \leq \frac{L \overline{\delta}}{\mu^{k-1}}.
% \end{align*}
Substituting $\norm{w_{N+k}^{(0)} - x_{k+1}}$ above in terms of $\frac{1}{\mu}$ from~\eqref{eqn:wk_0},
\begin{align}
    & \norm{K_{N+k}^{(d)} - (H_x (x_{k+1}))^{-1}}  \leq \frac{\rho_{N+k}^{d}}{l} (\frac{1}{2\mu} - \frac{\eta \gamma L \overline{\delta}}{2} \frac{1}{\mu^{k-1}}) \nonumber \\
    & + \eta \gamma \alpha^{(d-1)} \frac{1-(\rho_{N+k} \mu)^{d}}{\mu^{d-1}(1-\rho_{N+k} \mu)} \frac{L \overline{\delta}}{\mu^{k-1}}. \label{eqn:K_4}
\end{align}
From~\eqref{eqn:ipg_init}, we have
\begin{align}
    & \norm{K_{N+k+1}^{(0)} \hspace{-0.3em} - \hspace{-0.3em} (H_x (F(x_{k+1})))^{-1}} \nonumber \\
    & = \norm{K_{N+k}^{(d)} \hspace{-0.3em} - \hspace{-0.3em} (H_x (F(x_{k+1})))^{-1}} \leq \norm{K_{N+k}^{(d)} \hspace{-0.3em} - \hspace{-0.3em} (H_x (x_{k+1}))^{-1}} \nonumber \\
    & + \norm{(H_x (x_{k+1}))^{-1} \hspace{-0.3em} - \hspace{-0.3em} (H_x (F(x_{k+1})))^{-1}}. \nonumber
\end{align}
Upon substituting above from~\eqref{eqn:K_4} and using Lemma~\ref{lem:assump},
\begin{align}
    & \norm{K_{N+k+1}^{(0)} - (H_x (F(x_{k+1})))^{-1}} \nonumber \\
    & \leq \frac{\rho_{N+k}^{d}}{l} (\frac{1}{2\mu} - \frac{\eta \gamma L \overline{\delta}}{2} \frac{1}{\mu^{k-1}}) \nonumber \\
    & + \eta \gamma \alpha^{(d-1)} \frac{1-(\rho_{N+k} \mu)^{d}}{\mu^{d-1}(1-\rho_{N+k} \mu)} \frac{L \overline{\delta}}{\mu^{k-1}} + L_2 \norm{x_{k+1} - F(x_{k+1})}. \nonumber
\end{align}
Recall that $C_k = \norm{x_{k+1} - F(x_{k+1})}$. Since the state trajectory is a subset of $\O$ and $\O$ is bounded, $C_k < \infty$. 
% Since the state trajectory is a subset of $\O$ and $\O$ is bounded, $\exists C_0 > 0$ such that $\norm{x_k} \leq C_0$ for each $k \geq 0$. 
Then, from above, the L.H.S. of~\eqref{eqn:ind_clm_1} at $s=k+1$:
\begin{align}
    & \frac{\eta \gamma L \overline{\delta}}{2} \frac{1}{\mu^{k}} + l\norm{K_{N+k+1}^{(0)} - (H_x (F(x_{k+1})))^{-1}} \nonumber \\
    & \leq \frac{\eta \gamma L \overline{\delta}}{2} \frac{1}{\mu^{k}} + \rho_{N+k}^{d} (\frac{1}{2\mu} - \frac{\eta \gamma L \overline{\delta}}{2} \frac{1}{\mu^{k-1}}) \nonumber \\
    & + \eta \gamma l \alpha^{(d-1)} \frac{1-(\rho_{N+k} \mu)^{d}}{\mu^{d-1}(1-\rho_{N+k} \mu)} \frac{L \overline{\delta}}{\mu^{k-1}} + l C_k L_2. \label{eqn:mid_1}
\end{align}
We need to show R.H.S. above is upper bounded by $\frac{1}{2\mu}$, i.e.,
\begin{align*}
    & \frac{\eta \gamma L \overline{\delta}}{2} \frac{1}{\mu^{k}} + \rho_{N+k}^{d} (\frac{1}{2\mu} - \frac{\eta \gamma L \overline{\delta}}{2} \frac{1}{\mu^{k-1}}) \\
    & + \eta \gamma l \alpha^{(d-1)} \frac{1-(\rho_{N+k} \mu)^{d}}{\mu^{d-1}(1-\rho_{N+k} \mu)} \frac{L \overline{\delta}}{\mu^{k-1}} + l C_k L_2 \leq \frac{1}{2\mu} \\
    % & \iff \frac{\eta \gamma L \overline{\delta}}{2} \frac{1}{\mu^{k}} + \eta \gamma l \alpha^{(d-1)} \frac{1-(\rho_{N+k} \mu)^{d}}{\mu^{d-1}(1-\rho_{N+k} \mu)} \frac{L \overline{\delta}}{\mu^{k-1}} + 2 l C_0 L_2 \\
    % & \leq \frac{1}{2\mu}(1-\rho_{N+k}^{d}) + \rho_{N+k}^{d} \frac{\eta \gamma L \overline{\delta}}{2} \frac{1}{\mu^{k-1}} \\
    \iff & \frac{\eta \gamma}{2} \frac{1}{\mu^{k}} + \eta \gamma l \alpha^{(d-1)} \frac{1-(\rho_{N+k} \mu)^{d}}{\mu^{d-1}(1-\rho_{N+k} \mu)} \frac{1}{\mu^{k-1}} + \frac{l C_k L_2}{L \overline{\delta}} \\
    & \leq \frac{1}{2\mu L \overline{\delta}}(1-\rho_{N+k}^{d}) + \rho_{N+k}^{d} \frac{\eta \gamma}{2} \frac{1}{\mu^{k-1}}.
\end{align*}
If $\alpha^{(i)} < D_1 \frac{\mu^{i} (1-\mu \rho_{N+k})}{\eta \gamma l (1-(\mu \rho_{N+k})^{i+1})}$ for some $D_1 > 0$, then it is enough to show that
\begin{align*}
    & \frac{\eta \gamma}{2 \mu^k} + \frac{D_1}{\mu^{k-1}} + \frac{l C_k L_2}{L \overline{\delta}} \leq \frac{1}{2\mu L \overline{\delta}}(1-\rho_{N+k}^{d}) + \rho_{N+k}^{d} \frac{\eta \gamma}{2 \mu^{k-1}},
\end{align*}
which is equivalent to
\begin{align}
    & \frac{l C_k L_2}{L \overline{\delta}} \leq \frac{1}{2\mu L \overline{\delta}}(1-\rho_{N+k}^{d}) +  \frac{1}{\mu^{k-1}} (\rho_{N+k}^{d} \frac{\eta \gamma}{2} - D_1 - \frac{\eta \gamma}{2} \frac{1}{\mu}) \nonumber \\
    & = (1-\rho_{N+k}^{d}) (\frac{1}{2\mu L \overline{\delta}} - \frac{\eta \gamma}{2 \mu^{k-1}}) + \frac{1}{\mu^{k-1}} (\frac{\eta \gamma}{2} - \frac{\eta \gamma}{2 \mu} - D_1). \label{eqn:bd_11}
\end{align}

We prove~\eqref{eqn:bd_11} as follows.
Since $\rho_{N+k}<1$ and $d\geq 1$, $(1-\rho_{N+k}^{d}) > 0$. Moreover, 
\begin{align*}
    & \frac{\eta \gamma}{2} - \frac{1}{\mu} \frac{\eta \gamma}{2} - D_1 > 0 
    \iff \frac{1}{\mu} < 1 - \frac{2 D_1}{\eta \gamma}, \, D_1 < \frac{\eta \gamma}{2}.
\end{align*}
If $D_1 < \frac{\eta \gamma}{2} (1-\rho_{N+k})$, then $\rho_{N+k} <  1 - \frac{2 D_1}{\eta \gamma}$. Since $\rho_{N+k} < 1$, the above choice of $D_1$ also implies that $D_1 < \frac{\eta \gamma}{2}$. Similarly, for the first term on the R.H.S. of~\eqref{eqn:bd_11}, 
\begin{align*}
    & \frac{1}{2\mu L \overline{\delta}} - \frac{1}{\mu^{k-1}}  \frac{\eta \gamma}{2} > 0 
    \iff \frac{1}{\mu^{k-2}} < \frac{1}{\eta \gamma L \overline{\delta}}.
\end{align*}
Since $\mu > 1$, $\frac{1}{\mu} <  \frac{1}{\eta \gamma L \overline{\delta}}$ implies that the above holds. So, $D_1 < \frac{\eta \gamma}{2} (1-\rho_{N+k})$ and $\frac{1}{\mu} < \min \left\{1 - \frac{2 D_1}{\eta \gamma}, \frac{1}{\eta \gamma L \overline{\delta}}\right\}$ implies that the R.H.S. of~\eqref{eqn:bd_11} is positive. Thus, with $D_1$ and $\mu$ chosen as above and if~\eqref{eqn:bd_11} holds, then
\begin{align}
    \frac{\eta \gamma L \overline{\delta}}{2 \mu^k} + l \norm{K_{N+k+1}^{(0)} - (H_x (F(x_{k+1})))^{-1}} \leq \frac{1}{2\mu}. \label{eqn:K_5}
\end{align}
By definition of $\rho$, we have $\rho \geq \rho_{N+k}$. Thus, $\varrho < 1 - \rho$ implies that $\varrho < 1 - \rho_{N+k}$. Thus, with $D_1$ chosen as above, condition~(iii) implies that $\alpha^{(i)} < D_1 \frac{\mu^{i} (1-\mu \rho_{N+k})}{\eta \gamma l (1-(\mu \rho_{N+k})^{i+1})}$. Since $\rho \geq \rho_{N+k}$, $D_1 < \frac{\eta \gamma}{2} \varrho \implies D_1 < \frac{\eta \gamma}{2} (1-\rho) < \frac{\eta \gamma}{2} (1-\rho_{N+k})$. Thus, the condition~(iv) implies~\eqref{eqn:bd_11}. We have proved~\eqref{eqn:ind_clm_1} for $s=k+1$.

Finally, we prove~\eqref{eqn:ind_clm_2} at $s=k+1$ as follows, for which we will need~\eqref{eqn:K_5}.
From~\eqref{eqn:wk_1},
\begin{align}
    \norm{w_{N+k}^{(d)} - x_{k+1}} \leq \frac{\overline{\delta}}{\mu^{k}}. \label{eqn:wk_d1}
\end{align}
Upon substituting from~\eqref{eqn:newton_init} and~\eqref{eqn:meas_eqn},
\begin{align*}
    \norm{w_{N+k+1}^{(0)} - x_{k+2}} = \norm{F(w_{N+k}^{(d)}) - F(x_{k+1})}.
\end{align*}
From Lemma~\ref{lem:assump} and above, we have
\begin{align}
    \norm{w_{N+k+1}^{(0)} - x_{k+2}} \leq L \norm{(w_{N+k}^{(d)} - x_{k+1})}. \label{eqn:wk_plus_d}
\end{align}
Upon substituting above from~\eqref{eqn:wk_d1},
\begin{align}
    \norm{w_{N+k+1}^{(0)} - x_{k+2}} < \frac{L \overline{\delta}}{\mu^{k}}. \label{eqn:wk_01}
\end{align}
The above and~\eqref{eqn:K_5} implies that
\begin{align*}
    & \frac{\eta \gamma}{2} \norm{w_{N+k+1}^{(0)} - x_{k+2}} \\
    & + l \norm{K_{N+k+1}^{(0)} - (H_x (F(x_{k+1})))^{-1}} \leq \frac{1}{2\mu}.
\end{align*}
Having proved the condition above, we will again invoke Lemma~\ref{lem:ipg_conv}. To show that all the conditions of Lemma~\ref{lem:ipg_conv} hold, we denote $\rho_{N+k+1} := \sup_{i \geq 0} \norm{I - \alpha^{(i)} H_x (w_{N+k+1}^{(i)})}.$
Under Assumptions~\ref{assump:obsv}-\ref{assump:eigen} and Lemma~\ref{lem:assump} with $W = \D = B(x_{k+2},\frac{1}{\mu \eta \gamma})$, if $\alpha^{(i)} < \min \left\{\frac{1}{\Lambda},\frac{\mu^{i} (1-\mu \rho_{N+k+1})}{2l (1-(\mu \rho_{N+k+1})^{i+1})}\right\}$, the conditions of Lemma~\ref{lem:ipg_conv} holds. So, from Lemma~\ref{lem:ipg_conv} we have
\begin{align}
    \norm{w_{N+k+1}^{(d)} - x_{k+2}} \leq \frac{1}{\mu^{d}} \norm{w_{N+k+1}^{(0)} - x_{k+2}}. \label{eqn:w_2}
\end{align}
% Upon substituting from~\eqref{eqn:newton_init} and~\eqref{eqn:meas_eqn},
% \begin{align*}
%     \norm{w_{N+k+1}^{(0)} - x_{k+2}} = \norm{F(w_{N+k}^{(d)}) - F(x_{k+1})}.
% \end{align*}
% Under Assumption~\ref{assump:F_lips}, from above we have
% \begin{align*}
%     \norm{w_{N+k+1}^{(0)} - x_{k+2}} \leq L \norm{(w_{N+k}^{(d)} - x_{k+1})}.
% \end{align*}
Upon substituting from~\eqref{eqn:wk_plus_d} in~\eqref{eqn:w_2},
\begin{align*}
    \norm{w_{N+k+1}^{(d)} - x_{k+2}} \leq \frac{L}{\mu^{d}} \norm{(w_{N+k}^{(d)} - x_{k+1})}.
\end{align*}
If $d \geq 1 + \log_{\mu} L$,  then $\frac{L}{\mu^{d}} \leq \frac{1}{\mu}$. From above, then we have
\begin{align}
    \norm{w_{N+k+1}^{(d)} - x_{k+2}} \leq \frac{1}{\mu} \norm{(w_{N+k}^{(d)} - x_{k+1})}, \label{eqn:w_3}
\end{align}
which proves~\eqref{eqn:ind_clm_2} for $s=k+1$.

{\bf Part 2}: Next, we will prove that~\eqref{eqn:ind_clm_1}-\eqref{eqn:ind_clm_2} hold for $s=1$. We will follow the steps in Part~1 above, where our initial condition is~\eqref{eqn:init_cond_thm} from the theorem statement, instead of~\eqref{eqn:ind_clm_1}.

Our first step is to prove~\eqref{eqn:ind_clm_1} for $s=1$.
If $\alpha^{(i)} < \min \left\{\frac{1}{\Lambda},D_2 \frac{\mu^{i} (1-\mu_1 \rho_N)}{2l (1-(\mu \rho_N)^{i+1})} \right\}$ for some $D_2 > 0$, then following the steps in Part~1, with the condition~\eqref{eqn:init_cond_thm} instead of~\eqref{eqn:ind_clm_1}-\eqref{eqn:ind_clm_2}, we obtain a sufficient condition similar to~\eqref{eqn:bd_11}: it is enough to show that
\begin{align}
    & l C_0 L_2 \hspace{-0.2em}
    \leq \hspace{-0.2em} (1-\rho_N^{d}) \hspace{-0.2em} (\frac{1}{2\mu} - \frac{\eta \gamma \delta}{2}) \hspace{-0.2em} + \hspace{-0.2em} \delta (\frac{\eta \gamma}{2}  \hspace{-0.2em} - \hspace{-0.2em} \frac{\eta \gamma L \overline{\delta}}{2\delta} \hspace{-0.2em} - \hspace{-0.2em} l D_2), \label{eqn:bd_1}
\end{align}
which is condition~(v) in the theorem statement.
Since $\rho_N<1$ and $d\geq 1$, $(1-\rho_N^{d}) > 0$. Moreover,~\eqref{eqn:init_cond_thm} implies that $\frac{1}{2\mu} - \frac{\eta \gamma \delta}{2} > 0$. If $\alpha^{(i)} < \min \left\{\frac{1}{\Lambda}, \frac{\mu^{i} (1-\mu_1 \rho_{N+1})}{2l (1-(\mu \rho_{N+1})^{i+1})} \right\}$, then under the condition~\eqref{eqn:init_cond_thm}, Assumptions~\ref{assump:obsv}-\ref{assump:eigen}, and Lemma~\ref{lem:assump} with $W = \D = B(x_{1},\frac{1}{\mu \eta \gamma})$, the conditions of Lemma~\ref{lem:ipg_conv} holds. So, from Lemma~\ref{lem:ipg_conv} we have
\begin{align}
    \norm{w_{N}^{(d)} - x_{1}} \leq \frac{1}{\mu^{d}} \norm{w_{N}^{(0)} - x_{1}}, \label{eqn:w_0}
\end{align}
i.e., $\overline{\delta} < \frac{\delta}{\mu^d}$. If $d \geq \log_{\mu} L$,  then from above we have $\frac{\overline{\delta}}{\delta} < \frac{1}{L}$. For $\frac{\overline{\delta}}{\delta} < \frac{1}{L}$ and $D_2 < \frac{\eta \gamma (1 - L \frac{\overline{\delta}}{\delta})}{2l}$, the second term on the R.H.S. of~\eqref{eqn:bd_1} $\frac{\eta \gamma}{2} - \frac{\eta \gamma L C_4}{2} - l D_2 > 0$. Thus, with $D_2, d, \delta$ chosen as above and if condition~(v) holds, then 
\begin{align}
    \frac{\eta \gamma L \overline{\delta}}{2} + l \norm{K_{N+1}^{(0)} - (H_x (F(x_{1})))^{-1}} \leq \frac{1}{2\mu}, \label{eqn:K_6}
\end{align}
i.e.,~\eqref{eqn:ind_clm_1} holds for $s=1$. 

Next, to prove~\eqref{eqn:ind_clm_2} for $s=1$, we follow the steps similar to Part~1 of the proof after~\eqref{eqn:K_5}.  If $\alpha^{(i)} < \min \left\{\frac{1}{\Lambda}, \frac{\mu^{i} (1-\mu_1 \rho_{N+1})}{2l (1-(\mu \rho_{N+1})^{i+1})} \right\}$, then under the condition~\eqref{eqn:K_6}, Assumptions~\ref{assump:obsv}-\ref{assump:eigen}, and Lemma~\ref{lem:assump} with $W = \D = B(x_{2},\frac{1}{\mu \eta \gamma})$, the conditions of Lemma~\ref{lem:ipg_conv} holds. So, from Lemma~\ref{lem:ipg_conv} we have
\begin{align}
    \norm{w_{N+1}^{(d)} - x_{2}} \leq \frac{1}{\mu^{d}} \norm{w_{N+1}^{(0)} - x_{2}}. \label{eqn:w_4}
\end{align}
Upon substituting from~\eqref{eqn:newton_init} and~\eqref{eqn:meas_eqn},
\begin{align*}
    \norm{w_{N+1}^{(0)} - x_{2}} = \norm{F(w_{N}^{(d)}) - F(x_{1})}.
\end{align*}
From Lemma~\ref{lem:assump} and above, we have
\begin{align*}
    \norm{w_{N+1}^{(0)} - x_{2}} \leq L \norm{(w_{N}^{(d)} - x_{1})}.
\end{align*}
Upon substituting from above in~\eqref{eqn:w_4},
\begin{align*}
    \norm{w_{N+1}^{(d)} - x_{2}} \leq \frac{L}{\mu^{d}} \norm{(w_{N}^{(d)} - x_{1})}.
\end{align*}
If $d \geq 1 + \log_{\mu} L$,  then $\frac{L}{\mu^{d}} \leq \frac{1}{\mu}$. From above, then we have
\begin{align}
    \norm{w_{N+1}^{(d)} - x_{2}} \leq \frac{1}{\mu} \norm{(w_{N}^{(d)} - x_{1})},
\end{align}
i.e.,~\eqref{eqn:ind_clm_2} holds for $s=1$.

{\bf Part 3}: Using the principle of induction, from Part~1 and Part~2, we have proved that
\begin{align}
    \norm{w_{N+k}^{(d)} - x_{k+1}} \leq \frac{1}{\mu} \norm{w_{N+k-1}^{(d)} - x_{k}}, \, k \geq 0. \label{eqn:w_conv}
\end{align}
We have proved linear convergence of $w_{N+k}^{(d)}$. Next, we use this result~\eqref{eqn:w_conv} to show linear convergence of state estimates.

We propagate $w_{N+k}^{(d)}$ and $x_{k+1}$ forward $N-1$ times according to~\eqref{eqn:newton_est} and~\eqref{eqn:meas_eqn}, respectively:
\begin{align*}
    \norm{\hat{x}_{N+k} - x_{N+k}} & = \norm{F^{(N-1)}(w_{N+k}^{(d)}) - F^{(N-1)} (x_{k+1)}}.
\end{align*}
From Lemma~\ref{lem:assump} and above, we have
\begin{align*}
    \norm{\hat{x}_{N+k} - x_{N+k}} & \leq L^{N-1} \norm{w_{N+k}^{(d)} - x_{k+1}}.
\end{align*}
Upon substituting above from convergence of $w_{N+k}^{(d)}$ in~\eqref{eqn:w_conv},
\begin{align*}
    \norm{\hat{x}_{N+k} - x_{N+k}} & \leq \frac{L^{N-1}}{\mu} \norm{w_{N+k-1}^{(d)} - x_{k}}.
\end{align*}
Upon iterating the above from $k$ to $1$,
\begin{align}
    \norm{\hat{x}_{N+k} - x_{N+k}} & \leq \frac{L^{N-1}}{\mu^{k}} \norm{w_{N}^{(d)} - x_{1}}. \label{eqn:w_5}
\end{align}

The final step is to bound $\norm{w_{N}^{(d)} - x_{1}}$ above in terms of the initial error $\delta = \norm{w_{N}^{(0)} - x_{1}}$. It is achieved by using Lemma~\ref{lem:ipg_conv} and sufficiently large number of iterations $d$ at each sampling instant $k$ as follows. 
If $\alpha^{(i)} < \min \left\{\frac{1}{\Lambda}, \frac{\mu^{i} (1-\mu \rho)}{2l (1-(\mu \rho)^{i+1})} \right\}$, then under the condition~\eqref{eqn:init_cond_thm}, Assumptions~\ref{assump:obsv}-\ref{assump:eigen}, and Lemma~\ref{lem:assump} with $W = \D = B(x_{1},\frac{1}{\mu \eta \gamma})$, the conditions of Lemma~\ref{lem:ipg_conv} holds. So, from Lemma~\ref{lem:ipg_conv},
\begin{align*}
    \norm{w_{N}^{(d)} - x_{1}} \leq \frac{1}{\mu^{d}} \norm{w_{N}^{(0)} - x_{1}}. 
\end{align*}
Upon substituting from above in~\eqref{eqn:w_5},
\begin{align*}
    \norm{\hat{x}_{N+k} - x_{N+k}} & \leq \frac{L^{N-1}}{\mu^{k+d}} \norm{w_{N}^{(0)} - x_{1}}. 
\end{align*}
If $d \geq (N-1)\log_{\mu} L$,  then $\frac{L^{N-1}}{\mu^{d}} \leq 1$. From above,
\begin{align}
    \norm{\hat{x}_{N+k} - x_{N+k}} & \leq \frac{1}{\mu^{k}} \norm{w_{N}^{(0)} - x_{1}}. \label{eqn:x_conv}
\end{align}

From~\eqref{eqn:init_cond_thm}, we note that $\eta \gamma \delta < \frac{1}{\mu} < 1$. Also, $\frac{\overline{\delta}}{\delta} < \frac{1}{L}$. Then, $\frac{1}{\eta \gamma L \overline{\delta}} > \frac{1}{\eta \gamma \delta} > \mu > 1$. So, $\frac{1}{\mu} < \min \left\{1 - \varrho, \frac{1}{\eta \gamma L \overline{\delta}}\right\}$ is simplified to $\frac{1}{\mu} < 1 - \varrho$.
The proof is complete.
% Moreover, 
% \begin{align*}
%     & (\rho^{(d)} \frac{\eta \gamma}{2} - D_1 - \frac{\eta \gamma}{2} \frac{1}{\mu}) > 0 \\
%     \iff & \frac{1}{\mu} < \rho^{(d)} - \frac{2 D_1}{\eta \gamma}.
% \end{align*}
% Rearranging the terms,
% \begin{align*}
%     & \frac{1}{2\mu L \overline{\delta}}(1-\rho^{(d)}) +  \frac{1}{\mu^{k-1}} (\rho^{(d)} \frac{\eta \gamma}{2} - D_1 - \frac{\eta \gamma}{2} \frac{1}{\mu}) \\
%     & = (1-\rho^{(d)}) (\frac{1}{2\mu L \overline{\delta}} - \frac{1}{\mu^{k-1}}  \frac{\eta \gamma}{2}) + \frac{1}{\mu^{k-1}} (\frac{\eta \gamma}{2} - \frac{1}{\mu} \frac{\eta \gamma}{2} - D_1).
% \end{align*}
\end{proof}

Since $\mu > 1$, Theorem~\ref{thm:conv} implies that the sequence of estimates $\{\hat{x}_{k}, k\geq N\}$ in IPG observer, under the conditions of Theorem~\ref{thm:conv}, converges to the state trajectory of~\eqref{eqn:meas_eqn} at a {\em linear} rate. Below we discuss the implications and possible simplifications of the sufficient conditions~(i)-(v) above.

\begin{itemize}[wide, labelwidth=!, labelindent=0pt]
    \item {\bf Convergence rate}: Among the sufficient conditions in Theorem~\ref{thm:conv},~(i) requires the initialization of the estimate $w_N^{(0)}$ and the pre-conditioner $K_N^{(0)}$ to be within a certain neighborhood of the true initial state $x_1$ and $(H_x (x_{1}))^{-1}$, respectively. Thus, the convergence guarantee in Theorem~\ref{thm:conv} is local, similar to the Newton-type observers~\cite{moraal1995observer, biyik2006hybrida, biyik2006hybridb}. Particularly,~(i) characterizes a trade-off between this region of attraction and the convergence rate $\frac{1}{\mu}$. Moreover, from the conditions~(ii)-(iii) of Theorem~\ref{thm:conv}, $\rho < \frac{1}{\mu} < 1-\varrho$ characterizes a upper and lower bound on the convergence rate.
    \item {\bf Step-size}: Condition~(iii) states the step-size $\alpha^{(i)}$ should be small enough. We note that, by controlling the number of iterations $d$ (see (i)), $\frac{L \overline{\delta}}{\delta}$ ratio can be made smaller, leading to a larger upper bound for choosing the parameter $D_2$.
    Moreover, since $\mu > 1$ and $\mu \rho < 1$, for large enough iteration index $i$, condition~(iii) is simplified to $\alpha^{(i)} < \frac{1}{\Lambda}$.
    \item {\bf Condition (iv)}: The argument following~\eqref{eqn:K_4} required (iv)-(v) due to the pre-conditioner. The reason is as follows. From~\eqref{eqn:newton_init}, the initial estimate $w_{k+1}^{(0)}$ for the state at the next sampling instant $k+1$ is obtained by passing the previous estimate $w_k^{(d)}$ through the system dynamics $F$. On the contrary, from~\eqref{eqn:ipg_init}, the pre-conditioner $K_{k+1}^{(0)}$, which is an estimate of the inverse Jacobian $(H_x (x_{k+1}))^{-1}$ at the next sampling instant $k+1$, is obtained by directly setting it to the previous pre-conditioner $K_k^{(d)}$, which is an estimate of the the inverse Jacobian $(H_x (x_{k}))^{-1}$ at the previous instant. Since this initialization of pre-conditioner for the next instant is done without utilizing any information of the system~\eqref{eqn:meas_eqn}, the initial value of the pre-conditioner at any sampling instant need not result in a descent direction of the estimate update~\eqref{eqn:ipg_sq}. Thus, the additional error term regarding $C_k$ appears in the upper bound on the estimation error. By definition, $C_k$ is the difference between two consecutive system states at the sampling instant $k+1$. So, for asymptotically stable systems~\eqref{eqn:meas_eqn}, $\lim_{k \to \infty} C_k = 0$. Thus, for asymptotically stable systems,~(iv) will be easier to satisfy for larger values of sampling instant $k$, and will hold by default as $k \to \infty$ regardless of the other parameters in the inequality.
    \item {\bf Condition (v)}: It needs to be satisfied only at the first sampling instant. For example, one can have only the first sampling period small enough, so that the change in state $C_0 = \norm{x_1- x_0}$ is small. Thereafter, for $k>1$, the sampling period has no restriction.
\end{itemize}

{\bf Relation with Newton observer}:
Suppose the conditions in Theorem~\ref{thm:conv} hold. Since $\rho_{N+k} < 1$ and $\mu > 1$, from~\eqref{eqn:K_4}, $\lim_{d,k \to \infty} \norm{K_{N+k}^{(d)} \hspace{-0.3em} - \hspace{-0.3em} (H_x (x_{k+1}))^{-1}} = 0$, i.e., the preconditioner $K_{N+k}^{(i)}$ converges to the inverse Jacobian $H_x (x_{k+1}))^{-1}$ in the limit $i,k \to \infty$. So, the IPG observer~\eqref{eqn:K_sq}-\eqref{eqn:ipg_init} asymptotically replicates the Newton observer~\cite{moraal1995observer} under the above conditions, as hypothesized in~\cite{chakrabarti2023ipg}.

{\bf Relaxing Assumption~\ref{assump:eigen}}: 
Among the numerical examples in~\cite{chakrabarti2023ipg}, only the first and the third examples satisfy Assumption~\ref{assump:eigen}. However, the IPG observer converges in all three examples~\cite{chakrabarti2023ipg}. Thus, Assumption~\ref{assump:eigen} is not a necessary condition. In fact, this assumption can be removed by introducing an additional parameter $\beta > 0$ in the IPG observer. From Lemma~\ref{lem:assump}, the eigenvalues of $H^U_x$ are bounded. Then, one can choose $\beta > \max_{k \geq N} \mnorm{\lambda_{\min} \left[H^{U_k}_x \right]}$, in which case $\lambda_{\min} \left[H^{U_k}_x + \beta I \right] > 0, \forall k$. 
So, we replace the quantity $H^{U_k}_x (w_k^{(i)})$ in~\eqref{eqn:K_sq} with $H^{U_k}_x (w_k^{(i)}) + \beta I$. Regarding minimization of a cost $f$, from Theorem~1 in~\cite{chakrabarti2021accelerating}, we observe that {\em linear} convergence of IPG optimizer holds for non-zero $\beta$, without requiring strong convexity of $f$ as long as $\nabla^2 f + \beta I \succ 0$. Thus, instead of invoking the results from Theorem~1 in~\cite{chakrabarti2021accelerating} with $\beta = 0$, we can utilize Theorem~1 in~\cite{chakrabarti2021accelerating} with $\beta > \max_{k \geq N} \mnorm{\lambda_{\min} \left[H^{U_k}_x \right]}$; and modify the sufficient conditions in Theorem~\ref{thm:conv} by the additional non-zero parameter $\beta$, in accordance with the conditions~(8)-(9) from~\cite{chakrabarti2021accelerating}, to derive a similar result as Theorem~\ref{thm:conv} by following its presented proof in Section~\ref{sec:conv}. Since this analysis with $\beta \neq 0$ is similar, except accounting for $\beta$, we made the Assumption~\ref{assump:eigen} for a concise presentation.

\section{Conclusion}

We presented convergence analysis of the recently proposed IPG observer for a class of nonlinear systems in deterministic settings under standard assumptions in the literature. Upon utilizing a prior result, we established local {\em linear} convergence of the IPG observer by identifying a set of sufficient conditions. The practicality and possible simplifications of such conditions were discussed. Further, the IPG observer asymptotically approximates the Newton observer under certain conditions, as hypothesized in the IPG observer paper. As empirically shown in~\cite{chakrabarti2023ipg}, a key contribution of the IPG observer is its improved robustness against measurement noise than the fast Newton-type observers. Our future work will theoretically characterize the robustness of IPG observers against measurement noise. %The plan is first to analyze the robustness of the IPG optimizer for %convex problems and utilize that result in the IPG observer.

% \begin{figure}[htb!]
% \centering
% \includegraphics[width = 0.45\textwidth]{Full Version/spring_e-2_inset.jpg}
% \caption{\it State estimation error of different observers.}
% \label{fig:spring}
% \end{figure}

%                                   % the textheight of the last page by a suitable amount.
%                                   % This command does not take effect until the next page
%                                   % so it should come on the page before the last. Make
%                                   % sure that you do not shorten the textheight too much.

\bibliographystyle{unsrt}        
\bibliography{refs} 

\addtolength{\textheight}{-12cm}

\end{document}